\documentclass{amsart}
\pdfoutput=1
\usepackage{mathtools}
\usepackage{graphicx}
\usepackage{algorithm}
\usepackage{algpseudocode}
\usepackage{amsmath}
\usepackage{amsfonts}
\usepackage{verbatim}
\usepackage{framed}
\usepackage{todonotes}
\usepackage{booktabs}
\usepackage{subfigure}

\usepackage{color}
\definecolor{linkBlue}{rgb}{.1,.1, .6}
\definecolor{linkGreen}{rgb}{.1,.35, .1}
\usepackage{hyperref}
\hypersetup{linkbordercolor = .7 .7 1 , 
                  citebordercolor = .5 1 .5,
                  colorlinks = true, 
                  linkcolor = linkBlue,  
                  citecolor = linkGreen ,
                   linktocpage=true,
                  pdftoolbar=true} 

\usepackage[margin=1.4in]{geometry}

\usepackage{eqparbox}

\renewcommand\Re{\operatorname{\mathfrak{Re}}}

\newcommand{\half}{\frac{1}{2}}

\newcommand{\eqb}[1]{\begin{equation}\label{#1}}
\newcommand{\eqe}{\end{equation}}

\newcommand{\eqn}[2]{\begin{equation}\label{#1}#2\end{equation}}

\newcommand{\aln}[1]{\begin{align}#1\end{align}}
\newcommand{\splt}[1]{\begin{split}#1\end{split}}
\newcommand{\css}[1]{\begin{cases}#1\end{cases}}

\newcommand{\ra}{\rangle}
\newcommand{\la}{\langle}
\newcommand{\st}{\hbox{ \,\,subject to\,\, }}
\newcommand{\matb}{\left( \begin{matrix*}[r] }
\newcommand{\mate}{\end{matrix*}\right)}

\newcommand{\kp}{^{k+1}}
\newcommand{\km}{^{k-1}}
\newcommand{\opt}{^\star}
\newcommand{\reals}{\mathbb{R}}
\newcommand{\comps}{\mathbb{C}}

\newcommand{\ellone}{$\ell_1$}
\newcommand{\Chi}{\mathcal{X}}
\newcommand{\mcA}{\mathcal{A}}

\DeclareMathOperator*{\mmv}{MMV}
\DeclareMathOperator*{\minimize}{minimize\quad}

\DeclareMathOperator*{\sign}{sign}
\DeclareMathOperator*{\argmax}{arg\,max}
\DeclareMathOperator*{\argmin}{arg\,min}
\DeclareMathOperator*{\shrink}{shrink}
\DeclareMathOperator{\prox}{prox}
\DeclareMathOperator{\trace}{trace}
\newtheorem{theorem}{Theorem}

\theoremstyle{definition}

\begin{document}

\sloppy 

\bibliographystyle{ieeetr} 

\title[A Field Guide to Forward-Backward Splitting]{    A Field Guide to Forward-Backward Splitting  \\ with a FASTA implementation }
\author{Tom Goldstein, Christoph Studer, Richard Baraniuk}
\date{\today}

\maketitle

\begin{abstract}\vspace{-1.0cm}
Non-differentiable and constrained optimization play a key role in machine learning, signal and image processing, communications, and beyond.
For high-dimensional minimization problems involving large datasets or many unknowns, the forward-backward splitting method (also known as the proximal gradient method) provides a simple, yet practical solver.
Despite its apparent simplicity, the performance of the forward-backward splitting method is highly sensitive to implementation details. 

This article provides an introductory review of forward-backward splitting with a special emphasis on practical implementation aspects.  In particular, issues like stepsize selection, acceleration, stopping conditions, and initialization are considered. Numerical experiments are used to compare the effectiveness of different approaches.  

 Many variations of forward-backward splitting are implemented in a new solver called FASTA (short for Fast Adaptive Shrinkage/Thresholding Algorithm).  FASTA provides a simple interface for applying forward-backward splitting to a broad range of problems appearing in sparse recovery, logistic regression, multiple measurement vector (MMV) problems, democratic representations, 1-bit matrix completion, total-variation (TV) denoising, phase retrieval, as well as non-negative matrix factorization.

\end{abstract}

\tableofcontents

\section{Introduction}

A large number of non-differentiable and constrained convex optimization problems have the following form:
\eqb{general}
\minimize h(x) = f(x)+g(x),
\eqe 
where $x\in\reals^N ,$  $f$ is convex and differentiable, and $g$ is an arbitrary (i.e., not necessarily smooth) convex function.  Problems of the form \eqref{general} arise frequently in the many areas including  machine learning, dictionary learning \cite{JMOB10}, online learning for streaming data \cite{DS09}, sparse coding \cite{KL10,SGL12}, spectral clustering \cite{HS11}, sparse covariance estimation \cite{OONR12}, approximate nearest neighbor search and vector quantization~\cite{JFF11}, as well as compressive sensing \cite{FNW07}.
  
In many situations, the function $g$ is neither differentiable nor even finite-valued, in which case the problem \eqref{general} cannot be minimized using simple gradient-descent methods. However, for a large class of functions $g$ that arise in practice, one can efficiently compute the so-called \emph{proximal} operator 
\eqb{prox}
\prox_g(z,\tau) = \argmin_x \,\tau g(x)+\half \|x-z\|^2.
\eqe
The proximal operator finds a point close to the minimizer of $g$ without straying too far from a starting point $z$, and is often referred to as a \emph{backward} (or \emph{implicit}) gradient-descent step with stepsize $\tau.$
If the proximal operator~\eqref{prox} can be evaluated easily, then one can  solve \eqref{general} efficiently using the \emph{Forward-Backward Splitting} (FBS) method (also known as the proximal gradient method). Put simply, FBS can handle non-differentiable objectives and convex constraints while maintaining the simplicity of gradient-descent methods.

 Due to the vast applications of FBS and its utility for sparse coding and regression, many variants of FBS have been developed to improve performance and ease of use.  In its raw form, FBS requires the user to choose a number of convergence parameters that strongly effect both the performance and reliability of the algorithm.  These include stepsizes, stopping condition parameters, acceleration schemes, stability conditions, and initialization.   With the right modifications, FBS can be performed without substantial oversight from the user.  
 
 This article introduces  and reviews FBS from a practical point of view.  While this is not the first review article written on FBS, it differs from existing review articles in that it focuses on practical implementation issues.  For and excellent theoretical review of FBS, we refer to~\cite{CP11}.
 
 \subsection{Outline}
 
  In Section \ref{sec:fbs}, we introduce the forward-backward splitting method and discuss its convergence behavior.  Numerous example problems are discussed in Section \ref{sec:apps}, and for each we detail the formulation and solution via FBS.  In Section \ref{sec:baw}, we discuss practical issues related to the implementation of FBS.  The performance of variants of FBS on different test problems is explored in Section \ref{sec:num}.
 Practical issues of FBS are discussed in Section \ref{sec:baw} and are incorporated into a new reference implementation of FBS, called FASTA (short for Fast Adaptive Shrinkage/Thresholding Algorithm).  FASTA provides a simple interface for applying forward-backward splitting to a broad range of optimization problems.

 \subsection{A Word on Notation}
 The  $\ell_2$ norm is denoted by
 $\|x\| = \sqrt{\sum_i |x_i|^2}.$
We use the ``single-bar'' notation to define the $\ell_1$ norm $|x| = \sum_i |x_i|.$
 The transpose of a matrix/vector $x$ is denoted $x^T.$  The inner-product of two vectors $x$ and $y$ is denoted $\la x,y\ra = x^Ty.$  The inner product of two matrices $X$ and $Y$ is given by $\la X,Y\ra= \trace(X^TY) = \sum_{i,j} X_{i,j}Y_{i,j}.$ We use the symbol $\Re$ to denote the real part of a (possibly) complex quantity. For example $\Re x$ denotes the real part of $x.$  
The vector $x\opt$ that minimizes some objective function $h$ is denoted $x\opt = \argmin_x h(x).$   In contrast, the minimal {\em value} of $f$ is denoted $\min_x f(x).$

\section{Forward-Backward Splitting} \label{sec:fbs}
Forward-Backward Splitting is a two-stage method that addresses each term in \eqref{general} separately.  The FBS method is listed in Algorithm \ref{alg:fbs}.

\begin{algorithm}[H]
\begin{algorithmic}
 \While{not converged}
 \begin{flalign}
    \qquad \hat x^{k+1} &= x^k-\tau^k \nabla f(x^k) \label{fs}&\\
  \label{bs}   x^{k+1} &= \prox_g(\hat x^{k+1},\tau^k) = \argmin_x\,\, \tau^kg(x)+\frac{1}{2} \| x - \hat x^{k+1} \|^2&
  \end{flalign}
 \EndWhile
\end{algorithmic}
\caption{Forward-Backward Splitting}
\label{alg:fbs}
\end{algorithm}
Let's examine each step of the algorithm in detail.  Line \eqref{fs} performs a simple {\em forward} gradient descent step on $f$.  This step begins at iterate $x^k,$ and then moves in the direction of the (negative) gradient of $f,$ which is the direction of steepest descent.  The scalar $\tau^k$ is the {\em stepsize} which controls how far the iterate moves along the gradient direction during iteration $k.$

Equation \eqref{bs} is called the proximal step, or {\em backward} gradient descent step.   To understand this terminology, we examine the proximal operator \eqref{prox}. Any $x\opt$  that minimizes \eqref{prox} must satisfy the optimality condition
   \eqn{opt_g}{
   0=  \tau G+( x\opt - z)
   }
  where  $G\in \partial g(x\opt)$ is some sub-gradient (generalized derivative) of $g.$ Note that when $g$ is differentiable we simply have $G = \nabla g(x\opt).$ Equation \eqref{opt_g} rearranges to 
   $$x\opt = \prox_g(z,\tau) = z-\tau G.$$
    This shows that $x\opt$ is obtained from $z$ by marching down the sub-gradient of $g.$ For this reason, the proximal operator performs a gradient descent step.  Because the sub-gradient $G$ is evaluated at the final point $x\opt$ rather than the starting point $z,$ this is called {\em backward} gradient descent.  
    
    Equation \eqref{opt_g} is equivalent to the set inclusion
    $0 \in \tau \partial  g(x\opt)+( x\opt - z),$
    which rearranges to 
    $$ z \in \tau  \partial  g(x\opt)+ x\opt = (\tau  \partial  g+ I) x\opt.$$
    For this reason, the proximal operator \eqref{prox} is sometimes written 
    $$x\opt = (\tau  \partial  g+ I)^{-1} z = J_{\tau\partial g} z$$
    where $J_{\tau\partial g}=(\tau^k  \partial  g+ I)^{-1}$ is the {\em resolvent operator} of $\tau  \partial  g.$  The resolvent is simply another way to express the proximal operator. In plain terms, the proximal/resolvent operator, when applied to $z,$  performs a backward gradient descent step starting at $z$.
    
    Algorithm \ref{alg:fbs} alternates between forward gradient descent on $f$, and backward gradient descent on $g$.
    The use of a backward step for $g$ is advantageous in several ways. First, it may be difficult to choose a sub(gradient) of $g$ in cases where the sub-gradient $\partial g$ has a complex form or is not unique.  In contrast, it can be shown that problem \eqref{bs} always has a unique well-defined solution \cite{CW05}, and (as we will see later) it is often possible to solve this problem in simple closed form. Second, the backward step has an important effect on the convergence of FBS, which is discussed in the next section.
    
    \subsection{Convergence of FBS}
    The use of backward (as opposed to forward) descent for the second step of FBS is needed to guarantee convergence.  To see this,  let $x\opt$  denote a fixed point of the FBS iteration.  Such a point satisfies
       $$x\opt = prox_g\left(x\opt - \tau^k \nabla f(x\opt),\tau^k\right)= x\opt - \tau^k \nabla f(x\opt) - \tau^k G(x\opt)$$
    for some $G(x\opt) \in \partial g(x\opt).$  This simplifies to 
    $$0  =\nabla f(x\opt)+G(x\opt),$$
    which is the optimality condition for \eqref{general}.  This simple argument shows that a vector is a fixed-point of the FBS iteration if and only if it is optimal for \eqref{general}.  This equation has a simple interpretation:  when FBS is applied to an optimal point $x\opt,$ the point is moved to a new location by the forward descent step, and the backward descent step puts it back where it started.  Both the forward and backward step agree with one another because they both rely on the gradients evaluated at $x\opt.$ 
    
    This simple fixed point property is not enough to guarantee convergence.  FBS is only convergent when the stepsize sequence $\{\tau^k\}$ satisfies certain stability bounds.  An important property of FBS is that this stability condition does not depend on $g,$ but rather on the curvature of $f.$      
     In the case of a constant stepsize $\tau^k=\tau,$ FBS is known to converge for 
    \eqn{restrict}{\tau < \frac{2}{L(\nabla f)}}
    where $L(\nabla f)$ is a Lipschitz constant of $\nabla f$  (i.e., $\|\nabla(x)-\nabla(y)\|<L\|x-y\|$ for all $x,\,y$).   In many applications $f=\half \|Ax-b\|^2$ for some matrix $A$ and vector $b.$  In this case, $L(\nabla f)$ is simply the spectral radius of $A^TA.$
    
        For non-constant stepsizes, it is known that convergence is guaranteed if the stepsizes satisfy $0<l<\tau^k<u<2/L(\nabla f)$ for some upper bound $u$ and  lower bound $l$  (see theorem 3.4 in \cite{CW05} for this result and its generalization).
    
    In practice, one seldom has accurate knowledge of $L(\nabla f),$ and the best stepsize choice depends on both the problem being solved and the error at each iteration.  For this reason, it is better in practice to choose the sequence $\{\tau^k\}$ adaptively and enforce convergence using backtracking rules rather than the explicit stepsize restriction \eqref{restrict}.  These issues will be discussed in depth in Section \ref{sec:linesearch}.



\section{Applications of Forward-Backward Splitting} \label{sec:apps}
In this section we study a variety of problems, and discuss how they are formulated and solved using forward-backward splitting.
FBS  finds use in a large number of fields, including machine learning, signal and image processing, statistics, and communication systems.  Here, we briefly discuss a small subset of potential applications.  In Section~\ref{sec:num} we present numerical experiments using these test problems.

\newpage
\subsection{Simple Applications}
\subsubsection{Convex Constraints and the Projected Gradient Method}  \label{sec:projGrad}
The projected gradient (PG) method is a special case of FBS involving a convex constraint.
Suppose we are interested in solving 
 \eqn{cons}{\minimize f(x) \,\, \st x\in \mathcal{C}}
 for some convex set $\mathcal{C}$.  We can rewrite this problem in the form \eqref{general} using the (non-smooth) characteristic function $\Chi_\mathcal{C}(x)$ of the set~$\mathcal{C}$, which is zero for $x\in \mathcal{C}$ and infinity otherwise.  As a consequence, the  problem
   $$\minimize f(x) +\Chi_\mathcal{C}(x)$$
is then equivalent to \eqref{general} with $g(x) = \Chi_\mathcal{C}(x)$.  To apply FBS, we must evaluate the proximal operator of $\Chi_\mathcal{C}$:  
 \eqn{proxCon}{
 \prox_{\Chi_\mathcal{C}}(z,\tau) = \argmin_{x\in \mathcal{C}}  \half \|x-z\|^2.
 } 
The solution to \eqref{proxCon} is the element of $\mathcal{C}$ closest to $z$; this is simply the orthogonal projection of~$z$ onto the set $\mathcal{C}$. 
The resulting PG method was originally studied by Goldstein, Levitin, and Polyak \cite{Bertsekas76,Goldstein64,LP66}.

\subsubsection{Lasso Regression}
One of the earliest sparse regression tools from high-dimensional statistics is the Lasso regression, which is easily written in the form \eqref{cons}.

An important application of PG methods as in \eqref{cons} is the Lasso regression problem \cite{Tibshirani94}, which is an important sparse regression tool.  The Lasso regression is defined as follows:
\eqn{lasso}{
 \minimize   \half \| Ax-b \|^2
 \st \|x\|_1\le \lambda.
}
Here, the convex set $\mathcal{C} = \{x: \|x\|_1\le \lambda \}$ is simply an \ellone-norm ball.  
The corresponding  proximal step, i.e., projection onto the \ellone{}-norm ball, can be carried out efficiently using linear-time algorithms, such as the method proposed in \cite{DSSC08}. We note that FBS has been used previously for the Lasso regression problem~\eqref{lasso} in the  popular SPGL1 solver \cite{VF07,VF08}.

\subsubsection{{$\boldsymbol \ell_1$}-Norm Penalized (Logistic) Regression} \label{sec:l1ls}

One of the most common applications for FBS is the following \ellone{}-norm penalized least squares problem:
\eqn{bpdn}{
 \minimize  \mu \|x\|_1 + \half \| Ax-b \|^2.
}
In statistics, problem \eqref{bpdn} is used to find sparse solutions to under-determined least squares problems. 
Problem \eqref{bpdn}  is called basis pursuit denoising (BDPN) in the context of compressive sensing and sparse signal recovery  \cite{FNW07,KL10}.

Suppose the vector~$b\in\{0,1\}^M$ contains binary entries representing the outcomes of random Bernoulli trials.  The success probability of the $i$th entry is 
 $P(b_i=1\, |\,x)=e^{A_i x}/(1+e^{A_ix} )$ where $A_i$ denotes the $i$th row of $A.$
 One is then interested in solving the so-called sparse \emph{logistic} regression problem
\eqn{logistic}{
 \minimize  \mu \|x\|_1 + \mathrm{logit}(Ax,b)
}
with the logit penalty function defined as 
\eqn{loglink}{
\mathrm{logit}(z,b) =  \sum_{i=1}^{M} \log(e^{z_i}+1)  - b_iz_i. \notag
}

Both problems, \eqref{bpdn} and \eqref{logistic}, can be solved using FBS with $g(x) = \|x\|_1.$  The proximal operator  of $g$ is given by the well-known shrinkage operator $\shrink(z,\mu\tau),$ whose $i^\text{th}$ element  is obtained as  
 \eqn{shrink}  {
  \shrink(z,\mu\tau)_i = \shrink(z_i,\mu\tau)=   \sign(z_i) \max\{|z_i|-\mu\tau ,0\}.
  }

 \subsubsection{Multiple Measurement Vector (MMV)}
  In Section \eqref{sec:l1ls} we sought sparse vectors satisfying $Ax\approx b$ for some matrix $A$ and measurement vector $b.$  Suppose now that we have a matrix $B$ containing many measurement vectors, and we want to find a sparse matrix $X$ satisfying $AX\approx B.$   If we further suppose that all columns of $X$ must have the same sparsity pattern, then we arrive at the Multiple Measurement Vector (MMV) problem \cite{CREK05}.  To formulate MMV using convex optimization, we  need the group sparsity prior 
    $$\mmv(X) = \sum_i \sqrt{ \sum_j X_{ij}^2 } = \sum_i \|X_i\|$$
  where $X_i$ denotes the $i$th row of $X.$
  
  The convex formulation of MMV is
  $$\minimize_X  \mu \mmv(X)+\half \|AX-B\|^2$$
  where the scalar $\mu$ controls the strength of the penalty.
  This is easily solved using FBS.  The forward step is simply 
   $$\widehat X\kp = X^k- \tau^kA^T(AX^k-B).$$
  The backward step requires the proximal operator of $\mu \mmv(\cdot),$ which can be evaluated row-by-row.  The $i$th row is simply
    $$X\kp_i = \prox_{\mu \mmv}(\widehat X\kp, \tau^k )_i =  \widehat X\kp_i \frac{\max\{\|\widehat X\kp_i\|-\tau^k\mu,0\} }{\|\widehat X\kp_i\|}.$$

\subsubsection{Democratic Representations}

The dynamic range of signals plays an important role in approximate nearest neighbor search and vector quantization~\cite{JFF11}, as well as communications and robotics/control \cite{SGB14}.
Given a signal $b\in \reals^M$, a low-dynamic range representation can be found by choosing a suitable matrix $A\in \reals^{M\times N}$ with $M<N$, and by solving 
   \eqn{dem}{
     \minimize \mu \|x\|_\infty +  \half \| Ax-b \|^2.
   }
The problem \eqref{dem} often yields a so-called democratic representation $x\opt$ that has small dynamic range and  for which a large number of entries have equal (and small) magnitude \cite{SGB14}. The problem \eqref{dem} can, once again,  be solved using FBS.  The only missing ingredient is the proximal operator for the $\ell_\infty$-norm, which can be computed in linear time using the method in \cite{DSSC08}.

\subsubsection{Low-Rank (1-bit) Matrix Completion}

The goal of matrix completion is to recover a low-rank matrix $\hat X$ from a small number of linear measurements.  Such problems often involve the {\em nuclear norm} of a matrix, which is simply the sum of the magnitudes (i.e., the $\ell_1$ norm) of the eigenvalues.

  One prominent formulation of matrix completion takes the form
\eqn{matcomp}{\minimize \mu \| X \|_* + L(X,Y), }
where $ \| X \|_*$ is the low-rank inducing nuclear norm of the matrix $X$ and $L(\cdot,Y)$ is a loss function that depends on the model of the observed data $Y$ \cite{KMS09}.  
 
Suppose that $Y$ contains noisy and punctured (or missing) data of the entries of $\hat X$, and let $\Omega$ denote the set of indices that have been observed. In this case,  one can use 
  $$L(X,Y) = \sum_{\alpha\in\Omega} (X_\alpha - Y_\alpha)^2$$
as an appropriate loss function.
Another loss function arises in 1-bit matrix completion, which can be viewed as an instance of low-rank logistic regression  \cite{DPVW13}. Each observation $Y_\alpha$ is assumed to be a Bernoulli random variable with success probability 
 $P(Y_\alpha=1| \hat X_\alpha)=e^{\hat X_\alpha}/(1+e^{\hat X_\alpha}).$
For this case, the appropriate loss function is again the logit function: $$ L(X,Y)  = \mathrm{logit}(X,Y).$$  
 
For both quadratic and logit link functions, \eqref{matcomp} can be solved using FBS. Specifically, one needs to compute the proximal operator of the nuclear norm, defined as
\begin{align} \label{eq:nucnormproxy}
 \prox_*(Z,\tau) = \argmin_X \tau  \mu\|X\|_* + \half \|X-Z\|^2,
\end{align}
whose solution is given by the matrix $U \shrink(S,\mu\tau)V^T$, where  $Z=USV^T$ is the singular value decomposition (SVD) of $Z$ \cite{DPVW13}.

\subsection{More Complex Applications}    
Sometimes an objective function does not immediately decompose into a smooth part and a ``simple'' part  for which we can evaluate the proximal operator in closed form.  In this case, it is often possible to re-formulate the problem in such a way that  FBS is easily applied.  In this section, we look at two such problems (total-variation denoising and support vector machines) that can be reformulated using {\em duality} and then solved easily and efficiently  using FBS.  We also look at an example where the {\em convex relaxation} of a problem is easily solved using FBS.

\subsubsection{Total-Variation Denoising} \label{sec:tv}
Given a 2-dimensional image $u,$ the intensity of the pixel in the $i$th row and $j$th column is denoted $u_{ij}.$  The discrete gradient of $u$ is denoted $\nabla u.$  At each point in the image, the gradient is given by $(\nabla u)_{ij} = (u_{i+1,j}-u_{i,j},u_{i,j+1}-u_{i,j})^T$ and contains the discrete derivatives in the row and column directions.

 Given a noise-contaminated image $f,$ one can construct a denoised image by solving
\eqn{tv}{
\minimize_u \mu |\nabla u| + \half \|u-f\|^2 
}
where $\mu$ is a parameter to control the level of smoothing, and 
\eqn{tvdef}{
|\nabla u| =\sum_{i,j} \|(\nabla u)_{ij}\| =\sum_{i,j} \sqrt{ (u_{i+1,j}-u_{i,j})^2+(u_{i,j+1}-u_{i,j})^2}
} 
denotes the total-variation of $u$ \cite{ROF92}.  Rather than imposing sparsity on $u$ itself, the total variation regularizer imposes sparsity on the {\em gradient} of $u.$  As a result, problem \eqref{tv} finds a piecewise-constant approximation to $f.$

   To perform TV minimization using FBS, we must re-formulate $|\nabla u|$ into a simpler differentiable function.   We begin by letting $x_{ij}=(x_{ij}^r,x_{ij}^c)^T$ be a vector in $\reals^2.$  We can now write
    \eqn{tvmax}{
    \max_{\|x_{ij}\|\le1}  \la x_{ij}, (\nabla u)_{ij}\ra  =  \|  (\nabla u)_{ij}\|.
    }
    Equation \eqref{tvmax} follows from the Cauchy-Swartz inequality, which states that  $\la x_{ij}, (\nabla u)_{ij}\ra  \le \| x_{ij}\| \|  (\nabla u)_{ij}\|.$  Equality is attained by choosing $x_{ij}$ parallel to $(\nabla u)_{ij}$.  If we further choose $x_{ij}$ to have unit norm, then we arrive at \eqref{tvmax}.  If we apply \eqref{tvmax} to \eqref{tvdef}  we get ${|\nabla u| =  \max_{\|x\|_\infty \le1}  \la x, \nabla u\ra.}$ We now have
    \aln{
\min_u \mu |\nabla u| +\half \|u-f\|^2  &=    \min_u \max_{\|x\|_\infty \le1} \mu \la x, \nabla u\ra+ \half\|u-f\|^2 \nonumber \\
&=  \max_{\|x\|_\infty\le1} \min_u \mu\la x, \nabla u\ra+ \half \|u-f\|^2.  \label{tveq}
}
The inner minimization in \eqref{tveq} is now differentiable.  For a given $x,$ the minimal value of $u$ satisfies $u=f+\mu \nabla \cdot x,$ where $\nabla \cdot x$ is the discrete divergence (which is the negative adjoint of the gradient operator).  At pixel $ij,$ this operator takes on the scalar value ${(\nabla \cdot x)_{ij} = x^r_{i,j}-x^r_{i-1,j}+ x^c_{i,j}-x^c_{i,j-1}.}$  If we plug the optimal value of $u$ into \eqref{tveq} and simplify, we see that the optimal value of $x$ is given by:    
    \eqn{tvdual}{ 
    x^* = \argmax_{\|x\|_\infty\le1} -  \half \| \nabla \cdot x+\frac{1}{\mu}f\|^2= \argmin_{\|x\|_\infty\le1}  \half \| \nabla \cdot x+\frac{1}{\mu}f\|^2.
    }
This is simply a quadratic minimization with an infinity-norm constraint.  Problem \eqref{tvdual} is known as the {\em dual} form of \eqref{tv}.  This problem can be solved using FBS as in Section \eqref{sec:projGrad}.  The resulting algorithm alternately performs gradient descent steps on \eqref{tvdual}, and then re-projects the result back into the infinity-norm ball using the formula ${x_{ij} \leftarrow x_{ij}/\max\{\|x_{ij}\|,1\}.}$  Once problem \eqref{tvdual} has been solved, the optimal (denoised) image $u\opt=f+\mu \nabla \cdot x\opt$ is computed.  

This approach to total-variation minimization was first taken in \cite{Chambolle04} using constant stepsize parameters.  A similar approach was taken using accelerated variants of FBS in \cite{BT09:tv}.

\subsubsection{Support Vector Machines}  \label{sec:svm}
Consider a set of data points $\{d_i\}$ each of which lies in $\reals^n$ and has a binary label $l_i\in\{-1,1\}.$  The support vector machine (SVM) aims to find the hyper-plane in $\reals^n$ that separates the points with label $+1$ from the points with label $-1.$  This is done by solving 
  \eqn{svm}{
   \minimize_{w} \half\|w\|^2 + C\cdot h_l(Dw)
   }  
where $w\in \reals^n,$ $C$ is a constant regularization parameter, $D$ is a matrix with $i$th row equal to $d_i$, and $h$ is the {\em hinge loss} function 
 $$h_l(z) =\sum_i  \max\{1-l_i z_i,0\}.$$
 
 The hinge loss function penalizes points that either lie on the ``wrong'' side of the hyper-plane $w^Tx=0$ or that lie too close to this hyperplane.
  
  To apply FBS, we ``dualize'' this problem much like we did for total variation (Section \ref{sec:tv}).  We first note that for any $z\in \reals,$ we have $C\max\{z,0\} = \max_{0\le x\le C} xz.$
   Using this observation, we re-write \eqref{svm} as
  \eqn{svmpd}{\min_{w} \max_{0\le x\le C}   \half \|w\|^2 + x^T(1_n- LDw))=\max_{0\le x\le C}   \min_{w}  \half \|w\|^2 + x^T(1_n- LDw))}
  where the constraints on the vector $x$ are interpreted element-wise, $L$ is a diagonal matrix with $L_{i,i} = l_i,$ and $1_n$ is a vector of 1's.  This objective is quadratic in $w$, and the optimal value of $w$ is given by $D^TLx.$  If we plug this value in for $w$ in \eqref{svmpd} and multiply by -1 to convert the maximization for $x$ into a minimization, we arrive at the dual problem
     \eqn{svmdual}{\minimize_{0\le x\le C} \half \|D^TLx\|^2- \sum_i x_i.}
     This is a simple quadratic minimization with constraints of the form $0\le x_i \le C.$ The proximal operator corresponding to these constraints is the projection $\prox(z,t) = \min\{\max\{z,0\},C\}.$ Once the dual problem \eqref{svmdual} is solved using FBS, the primal solution is recovered from the formula $w=D^TLx.$  SVM's and the dual minimization are reviewed in~\cite{SS04}.

\subsubsection{Phase Retrieval and Rank Minimization} \label{phase}
A variety of non-convex problems can be relaxed into convex problems involving symmetric positive semi-definite (SPDP) matrices.  
This  idea was popularized by \cite{GW95}, in which it was shown that the NP-complete MaxCut problem has a convex relaxation with the form of a  semi-definite program. Here, we consider the more recent PhaseLift algorithm, which can be used for phase retrieval applications \cite{CSV13}.

Suppose we wish to recover a complex-valued signal $x\in \comps^N$ from measurements of the form $|\la a_i,x \ra|^2 = b_i$ for some set of vectors $\{a_i\}_{i=1}^M.$  In words, we take inner products of the vector $x$ with the measurement vectors $\{a_i\}$ and discard the phase information.

Recovery of $x$ from the phase-less measurements $b$ requires the solution of a system of non-linear equations.  However, one can relax the equations into a convex problem by defining $A_i = a_ia_i^T,$ and observing that $|\la a_i,x \ra|^2  = \la A_i,xx^T \ra = b_i.$  Letting $\mcA(xx^T)_i =\la A_i,xx^T \ra,$ the set of phase-less measurements  can be  written compactly as $\mcA(xx^T)=b.$   Finally, note that any SPSD rank-1 matrix $X$ can be factorized as $X=xx^T.$  This observation allows us to pose the recovery problem in the following form:
         \aln{\label{rankMin}
          \minimize \!\!\!\!\mathrm{rank}(X) \st  \!\! \mcA(X) = b, \, X\succeq 0.
          }
Even though $\mathrm{rank}(\cdot)$ is a non-convex function, one can form a  convex relaxation of \eqref{rankMin} by replacing $\mathrm{rank}(\cdot)$ with the sparsity-inducing nuclear norm.  The resulting problem, known as PhaseLift, is given by \cite{CSV13}:
      \aln{\label{rankMinrelaxed}
        \minimize\!\!\!\!\! \|X\|_*  \st  \!\! \mcA(X) = b, \, X\succeq 0.
          }
         
In the case where the measurement vector $b$ is contaminated by additive noise, we choose an  $\ell_2$-norm penalty, which stems from a Gaussian noise model. Specifically, we can write the resulting phase-retrieval problem as
        \eqn{smoothRankMin}{ 
          \minimize \!\!\! \mu\|X\|_* +\|\mcA(X) - b\|^2 \, \st  X\succeq 0, \notag
          }
which can be solved using FBS.  The relevant proximal operator corresponds to
$$\minimize \tau\mu\|X\|_* +\|X-Z\|^2,$$
whose solution is given by the matrix $U \shrink(\Lambda,\mu\tau)U^T$, where  $Z=U\Lambda U^T$ is the eigenvalue decomposition of $Z$.

\subsection{Non-convex Problems}
All applications discussed so far involved convex optimization.  In practice, FBS also works quite well when applied to {\em non-convex} problems.  However, in this case there are no theoretical guarantees that the method with converge to a global minimizer, or that the algorithm will even converge at all.   That being said, it is often the case for non-convex problems that no known method can guarantee optimality in polynomial time, and so this lack of theoretical guarantees makes FBS no weaker than other options.  

  When solving non-convex problems the user should be cautious of several things.  First, unlike in the convex case, the final result will be sensitive to the initial iterate.  For this reason, it is important to either initialize the methods with a good approximate solution, or else generate numerous solutions from different random initializations and choose the solution with the lowest objective value.    Second, because the method is no longer guaranteed to converge, the user must put an intelligent limit on the number of iterations.  In some situations, it may even be best to run the algorithm for a pre-determined number of iterations rather than using one of the precision-based termination conditions to be discussed in Section \ref{sec:stop}.

\subsubsection{Non-Negative Matrix Factorization}
Many problems in information science require a matrix to be factored into low-rank components.  Given a matrix $Q\in \reals^{M\times N},$ we seek $W\in \reals^{M\times K}$ and $C\in \reals^{N\times K}$ with 
  $$Q\approx WC^T.$$
  In {\em latent factor analysis,} the matrices $W$ and $C$ represent the underlying factors that ``explain'' the data.  Consider the performance of students on exam questions, each or which requires a blend of knowledge from $K$ different topics.  Suppose the vector $w_p\in \reals^K$ measures how much of each topic is required to solve problem $p,$ and  $c_s\in \reals^K$ describes a student's mastery of each of the $K$ topics.  If we choose these vectors appropriately, we expect student $s$ to score $Q_{ps}= \la w_p,c_s\ra$ on problem $p.$   Intuitively, it seems that a student's skill level should always be greater than or equal to zero, and each problem should require a non-negative amount of knowledge from each topic.  For this reason, is makes sense to recover the difficulty and skill vectors by solving
   \aln{\minimize_{W,C} &\|  Q-WC^T  \|^2  \label{nmf}\\
   \st & W\ge 0, C \ge 0 \nonumber
   }
   where the constraints on $W$ and $C$ are interpreted element-wise.
   
   This problem is known as {\em non-negative matrix factorization} (NMF) \cite{LS99}.  While various methods have been proposed to solve this problem \cite{LMACD14, LS00}, forward backward splitting remains one of the simplest.  The gradient of \eqref{nmf} is easily obtained using the chain rule, resulting in the forward step 
      $$
      \matb
      \widehat W\kp \\ \widehat C\kp
      \mate
      = \matb
       W^k\\ C^k
      \mate
      -   \tau^k \matb
        (W^k(C^k)^T-Q)C^k \\
          (W^k(C^k)^T-Q)^TW^k \\
      \mate.
      $$
The backward step is a trivial projection onto the set of non-negative matrices -- i.e., negative entries in the matrices are replaced by zeros.

\subsubsection{Max-Norm Optimization and the Max Cut Problem}
   Any $N\times N$ symmetric positive semidefinite matrix $A$ with rank at most $K$ can be written in the form $A=XX^T$ for some matrix $X\in \reals^{N\times K}$.  Given such a representation, the max-norm of $A$ is given by 
        $$\|A\|_{\max} = \max_i \|X_i\|^2_2 := \|X\|_{2,\infty}^2$$
    where $X_i$ is the $i$th row of $X.$  It can shown that the max-norm of a matrix is a good approximation to the nuclear norm (up to a constant factor), and can be used to promote low-rank solutions when the nuclear norm is not available because of the expense of computing singular value decompositions \cite{LRSST10}. 
	
An important application of the max-norm is for solving the max-cut problem on a graph \cite{LRSST10}. Consider a graph with $N$ vertices and an edge of weight $w_{ij}$ between vertex $i$ and $j.$  An  assignment of a binary label $x_i\in\{-1,1\}$ to each vertex is called a ``cut.''  The ``value'' of a cut is the sum of all edge weights connecting vertices with different labels.

Finding the graph cut with maximum value is NP-complete \cite{GW95}, however a good approximation can be found by solving     
        \aln{ \splt{\minimize_X   &  \la W,XX^T\ra\\  \label{maxcut}
   \st & \|X\|_{2,\infty}^2 \le 1
   }}
where $W$ is the weighted adjacency matrix of the graph, and $X\in \reals^{N\times K}$ is a matrix  \cite{LRSST10}. Once a solution to \eqref{maxcut} is found, a labeling is generated by computing $x_i = \sign(X_is)$ for some random vector $s.$  

The objective in \eqref{maxcut} differentiable, and the forward step of FBS is given by 
$$\widehat X\kp = X^k - \tau^k X(W^T+W).$$ 
The backward step is simply a projection onto the constraint set.  In this case, we have
  $$X\kp_i = \widehat X\kp_i \frac{\min\{\|\widehat X\kp_i\|,1\}}{\|\widehat X\kp_i\|}.$$


\section{Bells and Whistles} \label{sec:baw}
FBS in its raw form suffers from numerous problems including: (i) The convergence speed of FBS depends strongly on the choice of the stepsize parameters $\{\tau_k\}.$ The best stepsize choices may not be intuitively obvious.    (ii)  Convergence is only guaranteed when the stepsizes satisfy the stability condition \eqref{restrict} which depends on the Lipschitz constant for $\nabla f.$  In real applications one often has no knowledge of this Lipschitz constant and no practical way to estimate it.  (iii) To automatically stop the FBS iteration, a good measure of convergence is needed.

  In this section, we discuss practical methods for overcoming these problems.  We begin with adaptive schemes for automatic stepsize selection.  We then discuss backtracking methods that can guarantee stability even when the user has no explicit knowledge of $L(\nabla f).$  Finally, we discuss stopping conditions and measures of convergence.

\subsection{Adaptive Stepsize Selection} \label{sec:adapt}
The efficiency of FBS (and gradient methods in general) is very sensitive to the choice of the stepsize $\tau_k.$  For this reason, much work has been devoted to studying {\em adaptive} stepsize strategies. Adaptive methods automatically tune stepsize parameters in real time (as the algorithm runs) to achieve fast convergence.  In this section, we discuss spectral (also called Barzilai-Borwein) stepsize methods, and how they can be adapted to FBS.

Before considering the full-scale FBS method for \eqref{general}, we begin by considering the case $g=0.$  In this case $h=f$ and FBS reduces to simple gradient descent of the form
 \aln{x\kp = x^k-\tau^k \nabla f(x). \label{gradDescent}} 
     Spectral schemes for gradient descent were proposed by Barzilai and Borwein \cite{BB88}, who model the function $f$ as the simple quadratic function 
 \aln{ \label {bb_approx} f(x)   \approx   \hat f(x) =  \frac{a}{2} \|x\|^2+\la x , b \ra.}
 It can be shows that the optimal stepsize choice for \eqref{bb_approx} is $\tau = 1/\alpha.$  With this choice, the gradient descent method achieves a perfect minimizer of the simple quadratic \eqref{bb_approx} in one iteration.  This motivates the following stepsize scheme for gradient descent:  Before applying the descent step \eqref{gradDescent}, approximate $f$ with a quadratic of the form \eqref{bb_approx}.  Then, take a step of length $\tau^k=1/\alpha.$

Spectral stepsize rules have been generalized to handle FBS with specific choices of $g.$   When $g$ is the characteristic function of a convex set, the marriage of FBS with an adaptive stepsize produces the {\em spectral projected gradient method} (SPG) \cite{BMR00}.  SPG was further specialized to handle the case of $\ell_1$ constraints by  van den Berg and Friedlander \cite{VF07,VF08}.  The authors of \cite{WYG10} present a spectral stepsize rule that handles the case $g=|\cdot|$.  The resulting method solves $\ell_1$ penalizes least squares problems using a combination of spectral gradient descent and subspace optimization.  This idea was revisited in \cite{GS10}, in which the authors present a modification of FBS for $\ell_1$ penalized least squares that implicitly performs conjugate gradient optimization inside of low-rank subspaces. 

Spectral stepsizes for general FBS with arbitrary $f$ and $g$ were first studied in \cite{WNF09} for use with the solver SpaRSA.  The method presented here is a hybrid of the spectral scheme \cite{WNF09} with the ``adaptive'' stepsize rule presented in \cite{ZGD06}.

It was observed in the introduction that the stepsize restriction for FBS depends only on $f$ and not on $g.$  Spectral methods for FBS exploit this property.  The idea is to build a quadratic approximation for $f$ of the form \eqref{bb_approx} at each iteration, and then choose the optimal gradient descent stepsize  $\tau^k=1/a.$    Let 
  \aln{
  \Delta x^k &= x^k - x\km, \\
   \Delta F^k &= \nabla f (x^k) - \nabla f(x\km). 
  }
 If we assume a quadratic model for $f$ of the form \eqref{bb_approx}, then we have $$\Delta F^k= \nabla f(x^k) - \nabla f(x\km) = a(x^k - x\km).$$  The function \eqref{bb_approx} is fit to $f$ using a least squares method which chooses $a$ to minimize either $\| \Delta F^k - a \Delta x^k \|^2$ or $\| a^{-1}\Delta F^k -  \Delta x^k \|^2.$ Just like in the case $g=0,$ we then select a stepsize of length $\tau^k = 1/a$.  The resulting stepsize choice is given by  
  \aln{ \label{bbRules}
   \tau_s^k &=\frac{ \la   \Delta x^k, \Delta x^k \ra}{\la   \Delta x^k, \Delta F^k \ra},\text{ and }  \tau_m^k =\frac{\la   \Delta x^k, \Delta F^k \ra} { \la   \Delta F^k, \Delta F^k \ra}
   }  
respectively for each least squares problem.  The value $ \tau_s^k$ is known as the ``steepest descent'' stepsize, and  $ \tau_m^k$ is called the ``minimum residual'' stepsize \cite{ZGD06}.

A number of variations of spectral descent are reviewed by Fletcher \cite{Fletcher05}, several of which perform better in practice than the ``classic'' stepsize rules \eqref{bbRules}.  We recommend the ``adaptive'' BB method \cite{ZGD06},  which uses the rule
        \eqn{bbAdaptive}{ \tau^k=
    \css{
      \tau_m^k & \text{ if } \tau_m^k/\tau_s^k >\half \\
       \tau_s^k - \half  \tau_m^k & \text{ otherwise.}  
    }
    }
   Note that (particularly for non-convex problems), the stepsize $\tau_m^k$ or $\tau_s^k$ may be negative.  If this happens, the stepsize should be discarded and replaced with its value from the previous iteration.  When complex-valued problems are considered, it is important to use only the real part of the inner product in \eqref{bbAdaptive}.

\subsection{Preconditioning}


FBS can be ``preconditioned'' by multiplying the gradient directions of the forward step with a symmetric positive definite matrix $\Gamma.$  The backward step must be modified as well to ensure convergence.  The resulting approach, sometimes called the {\em proximal newton method} \cite{Bertsekas82, LSS12}, is described in Algorithm~\ref{alg:pfbs}.  
\begin{algorithm}[H]
\begin{algorithmic}
 \While{not converged}
 \begin{flalign}
    \qquad \hat x^{k+1} &= x^k-\tau^k \Gamma \nabla f(x^k) \label{pfs}&\\
  \label{pbs}   x^{k+1} &= \prox_g(\hat x^{k+1},\tau^k, \Gamma) = \argmin_x\,\, \tau^kg(x)+\frac{1}{2} \| x - \hat x^{k+1} \|^2_{\Gamma^{-1}}&
  \end{flalign}
 \EndWhile
\end{algorithmic}
\caption{Preconditioned Forward-Backward Splitting}
\label{alg:pfbs}
\end{algorithm} 
The preconditioned backward step \eqref{pbs} now requires the {\em generalized} proximal operator, which involves the following norm weighted by the matrix $\Gamma^{-1}$
  $$\|z\|^2_{\Gamma^{-1}} = \la \Gamma^{-1} z,z \ra.$$
For general $\Gamma$ the minimization \eqref{pbs} may not have a closed form solution, even when $g$ is ``simple.''  However, when the function $g$ acts separately on each element of $x,$ adding a {\em diagonal} preconditioner is a trivial modification.   For example, when $g=\mu |\cdot|$  (i.e., for the sparse least squares problems in section \ref{sec:l1ls}) and $\Gamma$ is diagonal, the $i$th entry of $ \prox_g(z ,\tau, \Gamma)$ is simply given by $\shrink(z_i,\mu\tau\Gamma_{ii}).$  Preconditioning with a diagonal $\Gamma$ is similarly easy in the case of the Support Vector Machine (Section \ref{sec:svm}).

   The preconditioner has another interpretation.  One can derive the preconditioned iteration (Algorithm \ref{alg:pfbs}) by making the change of variables $x\to Py$ for some invertible matrix $P$ and solving the problem 
     \eqn{precon}{\minimize_y f(Py)+g(Py).}
Each iterate $y^k$ can then be converted to an approximate solution to \eqref{general} by computing $x^k = Py^k.$ It can be shown that the iterates $\{x^k\}$ obtained through this procedure are equivalent to those obtained by Algorithm \ref{alg:pfbs} with preconditioner $\Gamma = P^TP.$ 

   The latter interpretation of the preconditioner is important because it suggests how to choose $\Gamma.$  When $f$ has the form $$f = \hat f (Ax)$$
for some matrix $A$ and function $\hat f,$  the preconditioned problem \eqref{precon} involves the objective $\hat f(APy).$  For example, when $f(x) = \|Ax-b\|^2,$ the problem \eqref{precon} contains the term $\|APy-b\|^2.$  If $A$ is poorly conditioned, then one should choose $P$ to be a {\em diagonal} matrix such that $AP$ is more well conditioned.  One effective method chooses the entries in $P$ so that all columns of $AP$ have unit norm.  The corresponding choice of $\Gamma = P^TP$ is then
  $$\Gamma_{ii} = \frac{1}{\|a_i\|^2}$$ 
  where $a_i$ denotes the $i$th column of $A.$
  
    Note that, because of the equivalence between preconditioned FBS and the original FBS on problem \eqref{precon},  all the tricks described in this section (including adaptivity and backtracking) can be applied to Algorithm \ref{alg:pfbs} as long as $\Gamma$ remains unchanged between iterations.  In the event that $\Gamma$ changes on every iteration, the converge theory becomes somewhat more complex (see \cite{CP11}). 

\subsection{Acceleration} \label{sec:fista}
Adaptive stepsize selection helps to speed the convergence of the potentially slow FBS method.  Another approach to dealing with slow convergence is to use predictor-corrector schemes that ``accelerate'' the convergence of FBS.
 
While several predictor-corrector variants of FBS have been proposed (see for example \cite{BF07_twist}), the algorithm FISTA has become quite popular because of its lack of tuning parameters and good worst-case performance.   FISTA relies on a one-size-fits-all sequence of acceleration parameters that works for any objective.   FISTA is listed in Algorithm \ref{alg:fista}.  Step 2 of FISTA simply performs an FBS step.  Step 4 performs {\em prediction}, in which the current iterate is advanced further in the direction it moved during the previous iteration.  The aggressiveness of this prediction step is controlled by the scalar parameter $\alpha^k.$  This acceleration parameter is updated in step 3, and increases on each iteration causing the algorithm to become progressively more aggressive.  The FISTA method alternates between ``predicting'' an aggressive estimate of the solution, and ``correcting'' this estimate using FBS to attain greater accuracy.

 One notable advantage of this approach is the worst-case convergence rate.  It has been shown that FISTA decreases the optimality gap (the difference between the objective at iteration $k$ and the optimal objective value)  with rate $O(\frac{1}{k^2})$ \cite{BT09}.  In contrast, the worst-case performance of conventional FBS is known to be $O(\frac{1}{k}).$
\begin{algorithm}[H]
\caption{FISTA}
\label{alg:fista}
\begin{algorithmic}[1]
\Require $ y^1=  x^0\in \reals^N, \, \alpha^1=1,\, \tau<1/L(\nabla G)$
\For {$k=1,2,3, \ldots $}
\State $x^k =  \prox_g  ( y^k-\tau \nabla f(y^k) , \tau)$
\State $\alpha^{k+1} = (1+\sqrt{1+4(\alpha^k)^2  })/2$
\State $y^{k+1} =x^k+\frac{\alpha^k-1}{\alpha^{k+1} }(x^k-x^{k-1})$\label{predict}
\EndFor 
\end{algorithmic}
\end{algorithm}
\noindent Rather than requiring the user to choose a stepsize, FISTA is generally coupled with a backtracking line search such as those described in the next section.

In theory, FISTA achieves superior worst-case performance by increasing the prediction parameter $\alpha^k$ on every iteration. In practice, however, FISTA does not always perform well when the prediction parameter $\alpha^k$ becomes large.  This is particularly true when a large/aggressive stepsize is used, in which case the objective values may oscillate rapidly before convergence is reached.  For this reason FISTA (and other Nesterov-type methods) generally perform better when a ``restart'' method is used  \cite{OC12}. Such methods allow $\alpha^k$ to increase while the performance of Algorithm \ref{alg:fista} is good, but reset the value to $\alpha^k=1$ if oscillations develop at iteration $k.$    
 
   One restart strategy simply sets $\alpha^k=1$ on iterations where the objective value increases.  However, this method requires the objective value to be calculated at each iteration, which is expensive and in general unnecessary.  The authors of \cite{OC12} suggest a restart method that does not require the objective.  The method restarts when the difference between iterates $x^k-x^{k-1}$ points in an ascent direction for the objective. It can be shown that this occurs whenever
      \eqn{restart}{\la y^k-x^k , x^k-x^{k-1}\ra \ge 0}
in which case the method should be restarted \cite{OC12}.  We will see in Section \ref{sec:stop} that the vector $y^k-x^k$ is parallel to the gradient of the objective function.  For this reason, condition \eqref{restart} restarts the method whenever the change between iterates forms an acute angle with the gradient (which is the direction of steepest ascent).

\subsection{Backtracking Line Search} \label{sec:linesearch}

Convergence of FBS can be guaranteed by enforcing the stability condition \eqref{restrict}.  In practice, though, the user generally has no knowledge of the global properties of $\nabla f,$ and so the actual value of this stepsize restriction is unknown.  In Section \ref{sec:adapt} we discuss adaptive methods that automatically choose stepsizes for us.  This does not free us from the need for stability conditions: spectral methods are, in general,  not guaranteed to converge, even for convex problems \cite{BMR00}.  

Even without knowing the stepsize restriction \eqref{restrict},  convergence can be guaranteed by incorporating a \emph{backtracking line search}.   Such methods proceed by checking a line search condition after each iteration of FBS.   The search condition usually enforces that the objective has decreased sufficiently. If this condition fails to hold, then {\em backtracking} is performed -- the stepsize is decreased and the FBS iteration repeated until the backtracking condition (i.e., sufficient decrease of the objective) is satisfied. 
 
Line search methods were originally proposed for smooth problems \cite{BMDIC95}, and later for projected gradient schemes \cite{KM96,BMR00}.  A backtracking  scheme for general FBS was proposed by Beck and Teboulle for use with FISTA \cite{BT09}.  
  Many authors consider these line searches to be overly conservative, particularly for poorly conditioned problems \cite{BMR00}, and prefer non-monotone line search conditions \cite{GLL86,ZH04}.  Rather than insisting upon an objective decrease on every iteration, non-monotone line search conditions allow the objective to increase within limitations.
  
    Non-monotone line search methods are advantageous for two reasons. First, for poorly conditioned problems (i.e., objective functions with long, narrow valleys) it may be the case that iterate $x\kp$  lies much closer to the minimizer than $x^k,$ despite $x\kp$ having the larger objective value. A non-monotone line search prevents such iterates from being rejected.  

Second, the objective function must be evaluated every time the line-search condition is tested.  For complex problems where evaluating the objective is costly and multiple backtracking steps are necessary for each iteration,  the line search procedure may dominate the runtime of the method.  Non-monotone line search conditions are less likely to be violated and thus backtracking terminates faster, which alleviates this computational burden.

The method proposed here is inspired by the monotone search proposed in \cite{BT09} for general FBS, and generalizes the non-monotone strategy of \cite{GLL86} for SPG.  
    
     Let $M>0$ be an integer line search parameter, and define
  $$ \hat f^k =  \max \{ f^{k-1},f^{k-2},\dots,f^{k- \min\{M,k\} }  \} .$$
After each step of FBS, the following line search condition is checked.~\
  \eqn{lscon}{
    f(x\kp)<  \hat f^k + \Re  \la  x\kp -x^k, \nabla f(x^k) \ra\\+ \frac{1}{2\tau^k}  \| x\kp-x^k\|^2.
    } 
    If the backtracking condition~\eqref{lscon} fails, the stepsize is decreased until~\eqref{lscon} is satisfied. This process is formalized in Algorithm~\ref{alg:search}.   Note that Algorithm \ref{alg:search} always terminates because condition \eqref{lscon} is guaranteed to hold whenever $\tau^k$ is less than the reciprocal of the Lipschitz constant of $\nabla f.$
    
    \begin{algorithm}[t]
\caption{Non-Monotone Line Search }
\label{alg:search}
\begin{algorithmic}
\While { $x^k$ and $x\kp$ violate  condition \eqref{lscon} } 
\State $\tau^k \gets \tau^k/2$
\State $x\kp \gets   \prox_g  ( x^k-\tau \nabla f(x^k) , \tau)$
\EndWhile
\end{algorithmic}
\end{algorithm}

A convergence proof for FBS with the line search described in Algorithm \ref{alg:search} is given in the appendix.   

\subsection{Continuation}
Some types of regression problems are solved very quickly when solutions are highly sparse, but become numerically difficult when solutions lack sparsity. This is sometimes the case for problem \eqref{bpdn}.  When $\mu$ is large,  solutions to \eqref{bpdn} are highly sparse.  For small $\mu,$ solutions may lack sparsity, resulting in slow convergence if $A$ is poorly conditioned.

    Continuation methods for  \eqref{bpdn} exploit this observation by choosing a large initial value of $\mu$ and  decreasing this parameter over time.  This way, the solver can use the results of ``easy'' problems as a warm start for more difficult problems with less sparsity.   Furthermore, using a warm start keeps the support of each iterate small, whereas the support might ``blow up'' for many iterations if difficult problems are attacked with a bad initializer.   
    
    Continuation techniques have been proposed by a number of authors (see for example \cite{FNW07,BBC09}) and are a keystone component in methods such as fixed-point continuation \cite{HYY07} that aim to approximate solutions to under-determined problems of the form 
      \eqn{bpdn_eq}{
 \minimize  \mu \|x\|_1 \st    Ax-b = 0
}
by solving \eqref{bpdn} with very small $\mu.$
     
     When choosing values for $\mu,$ it helps to observe that the solution to \eqref{bpdn} is zero when $\mu\ge \|A^T b\|_\infty.$  For this reason it is suggested to choose an initial parameter of $\mu =  \eta \|A^T b\|_\infty$ for some $\eta<1.$  After we obtain a solution using the current value of $\mu,$ we replace $\mu \leftarrow \eta \mu$ and continue iterating until the desired value of $\mu$ is reached.  We suggest to choose $\eta = \frac{1}{5},$ although the practical performance of continuation is not highly sensitive to this parameter. 
     
     Continuation is mostly effective for problem \eqref{bpdn} when $A$ is poorly conditioned, or when extremely large regularization parameters are needed.  Note that continuation does not always enhance performance, and may sometimes even make performance dramatically worse.

\subsection{Stopping Conditions}\label{sec:stop}
While the accuracy of FBS becomes arbitrarily good as the number of iterations approaches infinity, we must of course stop after a finite number of iterations.  A good stopping criteria should be strict enough to guarantee an acceptable degree of accuracy without requiring an excessive number of iterations.  

Our stopping conditions will be based on the {\em residual}, which is simply the derivative of the objective function (or a sub-gradient in the case that $g$ is non-differentiable).  Because $f$ is assumed to be smooth, we can differentiate this term in the objective directly.  While we may not be able to differentiate $g$, we see from equation \eqref{opt_g} that a sub-gradient is given by $(\hat x^{k+1}-x\kp)/\tau^k \in \partial g(x\kp).$  We now have the following formula for the residual $r\kp$ at iterate $x\kp:$
  \eqn{resid}{r\kp = \nabla f(x\kp)+\frac{ \hat x^{k+1}-x\kp}{\tau^k}.}

A simple termination rule would stop the algorithm when $\|r\kp\|<tol$ for some small tolerance $tol>0$.  However, this rule is problematic because it is not {\em scale invariant}.   To understand what this means, consider the minimization of some objective function $h(\cdot).$  This function can be re-scaled by a factor of 1000 to obtain $\hat h = 1000h.$  The new rescaled objective has the same minimizer as the original, however the sub-gradient $\partial \hat h(x^k)$ is 1000 times larger than $\partial  h(x^k).$  A stopping parameter $tol$ may be reasonable for minimizing $h$ but overly strict for minimizing $\hat h,$ even though the solutions to the problems are identical.   Ideally, we would like scale invariant stopping rules that treat both of these problems equally.

One way to achieve scale invariance is by replacing the residual with the {\em relative residual}. To define the relative residual, we observe that \eqref{resid} is small when
  \eqn{relmatch}{ \nabla f(x\kp)\approx - \frac{ \hat x^{k+1}-x\kp}{\tau^k}.}
In plain words, the residual \eqref{resid} measures the difference between the gradient of $f$ and the negative sub-gradient of $g$. The relative residual $r\kp_r$ measures the {\em relative} difference between these two quantities, which is given by
\eqn{rr}{ r\kp_r =  \frac{ \| \nabla f(x\kp)+\frac{ \hat x^{k+1}-x\kp}{\tau^k}   \| }{ \max\{
\|\nabla f(x\kp)\|,
\|\frac{ \hat x^{k+1}-x\kp}{\tau^k}\|
 \}+\epsilon_r}
 =  \frac{ \| r\kp\|   }{ \max\{
\|\nabla f(x\kp)\|,
\|\frac{ \hat x^{k+1}-x\kp}{\tau^k}
 \|\}+\epsilon_r}   }
where $\epsilon^r$ is some small positive constant to avoid dividing by zero.

Another more general scale invariant stopping condition uses the normalized residual, which is given by 
\eqn{rn}{r\kp_n = \frac{\|r\kp\|}{\|r^1\|+\epsilon_n}}
where the small constant  $\epsilon_n$ prevents division by zero. Rather than being an absolute measure of accuracy, the normalized residual measures how much the approximate solution has improved relative to $x^1.$

Both scale invariant conditions have advantages and disadvantages.  The relative residual works well for a wide range of problems and is insensitive to the initial choice of $x^0.$  However, the relative residual looses scale invariance when $\nabla f(x\opt)=0$ (in which case the denominator of \eqref{rr} nearly vanishes).  This happens, for example, when $g$ is the characteristic function of a convex set and the constraint is inactive at the optimal point (i.e., the optimal point lies in the interior of the constraint set).  In contrast, the normalized residual can be effective even if $\nabla f(x\opt)=0.$  However, this measure of convergence is potentially sensitive to the choice of the initial iterate, and so it requires the algorithm to be initialized in some consistent way.  The strictness of this condition can also depend on the problem being solved.

For general applications, we suggest a combined stopping condition that terminates the algorithm when either $r\kp_r$ or  $r\kp_n$ gets small.  In this case, $r\kp_r$ terminates the iteration if a high degree of accuracy is attained, and $r\kp_n$ terminates the residual appropriately in cases where $0\in \partial g(x\opt).$

\section{FASTA:  A Handy Forward-Backward Solver}
To create a common interface for testing different FBS variants, we have created the solver FASTA (Fast Adaptive Shrinkage/Thresholding), which implements forward-backward splitting for arbitrary problems. Many improvement to FBS are implemented in FASTA including adaptively, acceleration, backtracking, a variety of stopping conditions, and more.  FASTA enables different FBS variants to be compared objectively while controlling for the effects of stepsize rules, programming language, and other implementation details.  

Rather than addressing the problem \eqref{general} directly, FASTA solves general problems of the form
\eqb{general_fasta}
\minimize h(x) =  \tilde f(Ax)+g(x),
\eqe 
where $\tilde f$ and $A$ are chosen so that $\tilde f(Ax) = f(x).$  This problem form allows FASTA to compute $Ax^k$ once per iteration, and use the result twice to evaluate both the objective (which requires $\tilde f(Ax)$), and its gradient (given by $A^T\nabla \tilde f(Ax)$). The user can solve an arbitrary problem by supplying $A,$ the gradient of $f$ and the proximal mapping (i.e., backward gradient descent operator) for $g$.  However, custom wrappers are provided that solve all of the test problems described in this article, and codes are provided to reproduce the results in Section \ref{sec:num}.

\section{Numerical Experiments} \label{sec:num}

We compare several variants of FBS and compare their performance using the test problems described in Section \ref{sec:apps}.  The variants considered here are the original FBS with constant stepsize, and the accelerated variant FISTA described in Section \ref{sec:fista}.  We also consider the adaptive method described in Section \ref{sec:adapt}, which is implemented in the solver FASTA.  This method uses the spectral stepsize rules proposed in \cite{WNF09} and \cite{ZGD06}.

The stepsize parameter for all methods was initialized by first estimating the Lipschitz constant $L(\nabla f).$ Two random vectors $x^1,x^2$ were generated, and the estimate $ \tilde L= \|\nabla f(x^2) - \nabla f (x^1)\|/ \| x^2-x^1 \| \le  L(\nabla f)$ was formed. The initial stepsize was then $\tau^0 = 10/\tilde L.$  This stepsize is guaranteed to be at least an order of magnitude larger than the true stepsize restriction for FISTA ($1/L$).  The backtracking scheme in Section \ref{sec:linesearch} was then used to guarantee convergence.  We found that this implementation enables faster convergence than implementations that explicitly require a Lipschitz constant for $\nabla f$.

 At each iteration of the algorithms considered, the relative residual \eqref{rr} was computed and used as a measure of convergence.  Time trials for all methods were terminated on the first iteration that satisfied $r\kp_r<10^{-4}.$ If this stopping condition was not reached, iterations were terminated after 1000 iterations (except for the SVM problem, which was allowed up to 5000 iterations).  Time trial results are averaged over 100 random trials.
 
\begin{figure}[t]
\vspace{-0.2cm}
\centering
\subfigure[Matrix Completion]{\includegraphics[width=0.45\columnwidth, trim = .5cm 7.5cm 1cm 7.5cm, clip]{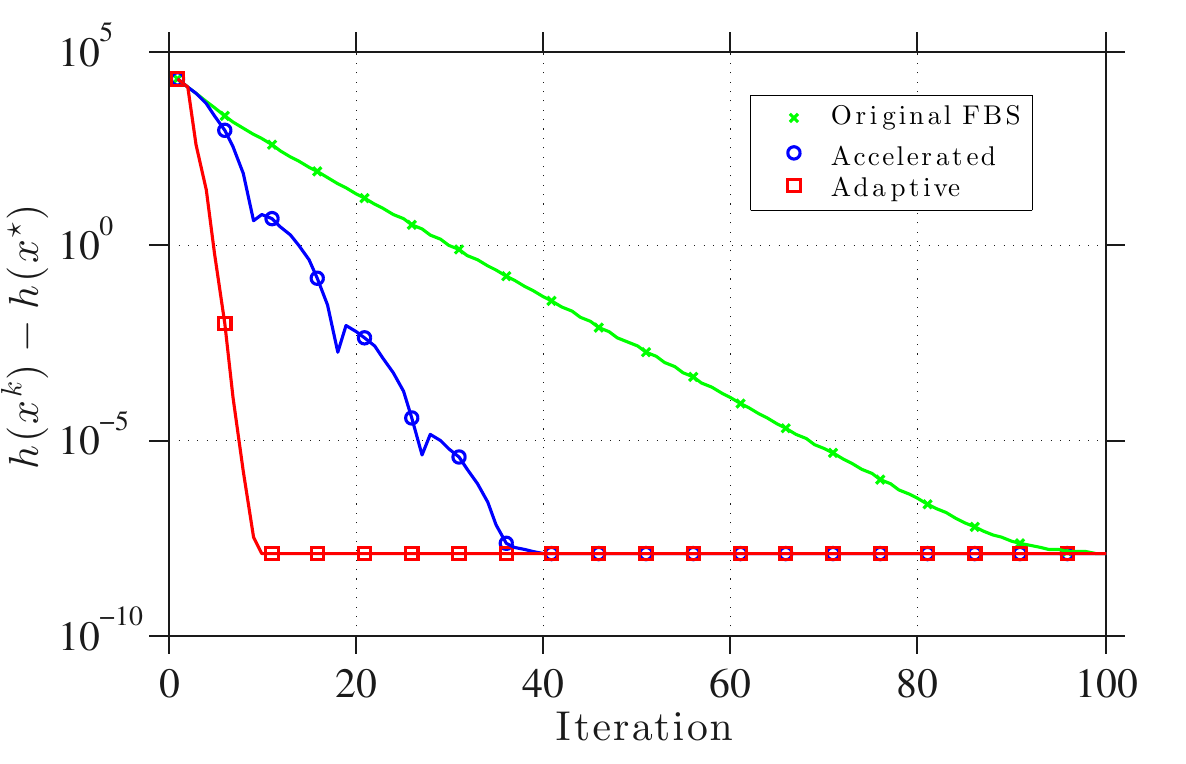}}
\hspace{0.9cm}
\subfigure[Support Vector Machine]{\includegraphics[width=0.45\columnwidth, trim = .5cm 7.5cm 1cm 7.5cm, clip]{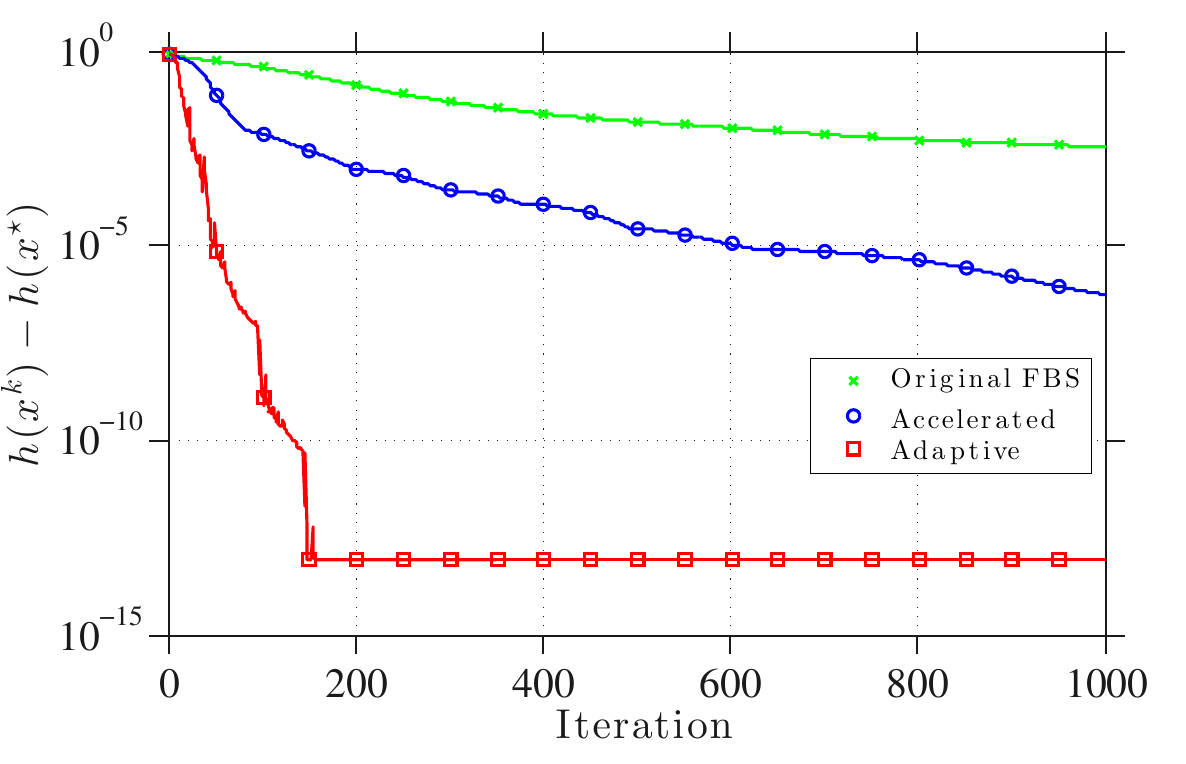}}\vspace{-0.2cm}
\subfigure[Non-negative Matrix Factorization]{\includegraphics[width=0.45\columnwidth, trim = .5cm 7.5cm 1cm 7.5cm, clip]{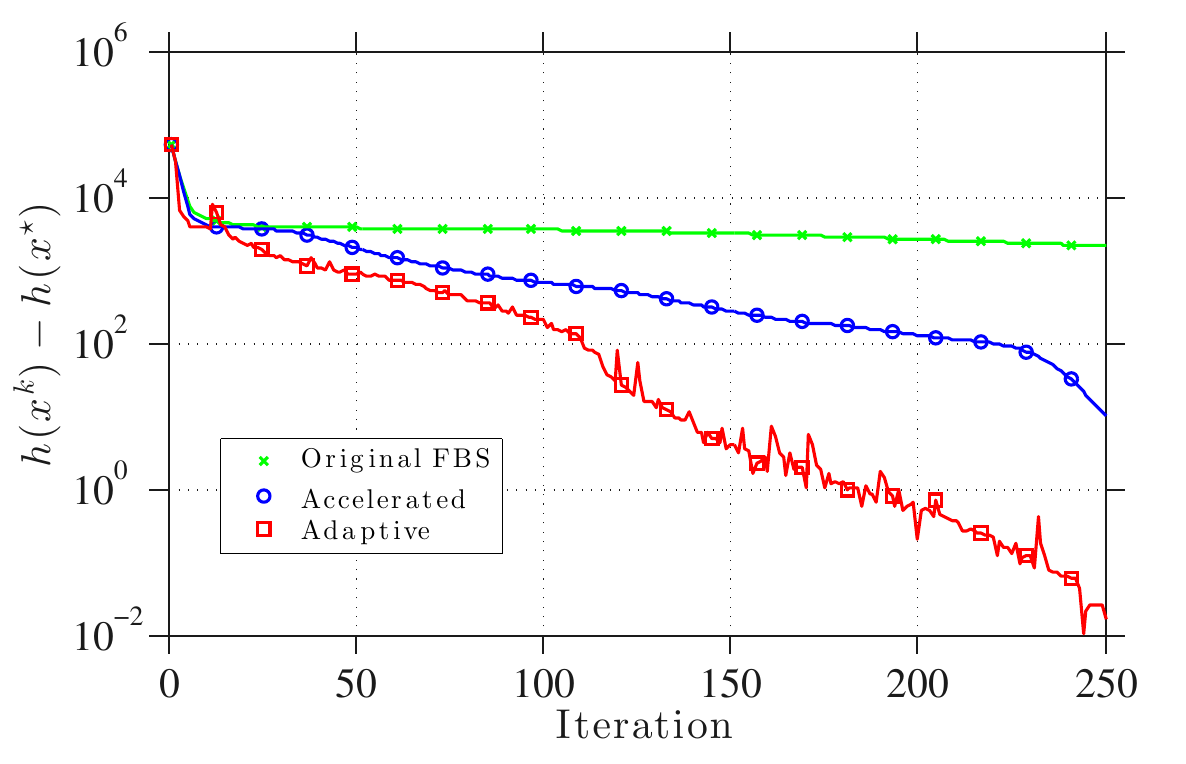}}
\hspace{0.9cm}
\subfigure[Total Variation]{\includegraphics[width=0.45\columnwidth, trim = .5cm 7.5cm 1cm 7.5cm, clip]{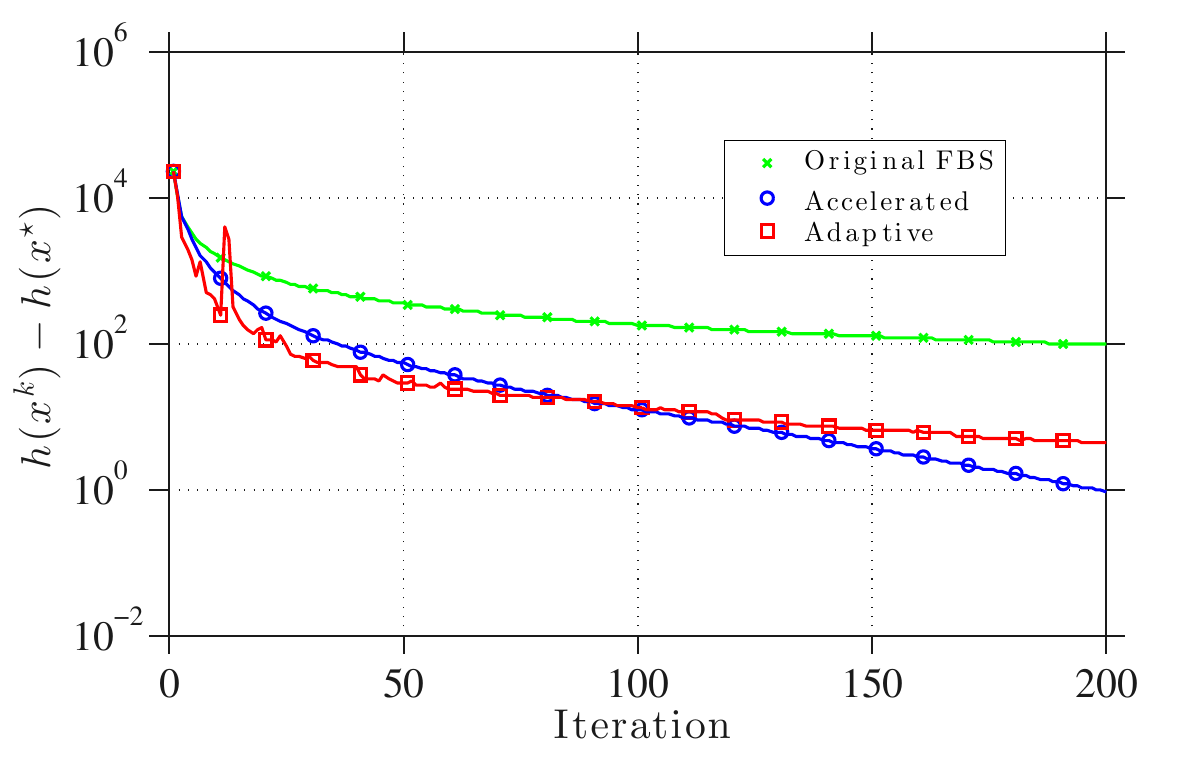}}
\vspace{-0.1cm}
\caption{Sample converge curves for FBS, FBS+acceleration (FISTA), and FBS+adaptivity (SpaRSA) for four diverse test problems (see Section \ref{sec:apps} and \ref{sec:num} for details). The vertical axis shows the optimality gap. }
\label{fig:curves}
\vspace{-0.3cm}
\end{figure} 

\newpage
\subsection{Test Problem Details}

\paragraph{\textbf{Lasso:}}\label{sec:lassoRes}

We generated problems of the form \eqref{lasso} using random Gaussian matrices $A\in\reals^{M\times N}$ with $N=1000.$ Experiments were done with both $M=500$ (moderately under-sampled) and $M = 100$ (highly under-sampled).  The true signal $x^0$ had $20$ entries of unit magnitude; the remaining entries were zero.   The indices of the support set were chosen at random.    The measurement vector was generated using the formula $b=Ax^0$ and then contaminated by additive Gaussian noise to achieve an signal-to-noise ratio (SNR) of 13\,dB. Recovery was  performed with $\lambda = 15.$

\paragraph{\textbf{${\boldsymbol\ell_1}$-Norm Penalized Least Squares:}}

We generated problems of the form \eqref{bpdn} using a similar procedure as for Lasso. This time, however, with the noise scaled to achieve $\text{SNR}=20$\,dB. Recovery was performed with $\mu = 0.1.$

\paragraph{\textbf{Logistic Regression:}}
 
We generated random matrices $A\in\reals^{M\times N}$ of dimension  $N=1000$ and $M=500$ with zero-mean Gaussian entries of variance of $4$.  Once again, the support of the true signal comprised 20 randomly chosen elements.  Given $z = Ax,$ $b$ was constructed as a set of random Bernoulli variables with probability of success given by the logistic function $P(b_i=1)=e^{z_i}/(1+e^{z_i}).$  The problem \eqref{logistic} was solved with $\mu = 20.$

\paragraph{\textbf{Multiple Measurement Vector:}}
We use the synthetic test problem suggested by \cite{CREK05}.  A random Gaussian measurement matrix $A\in \reals^{20\times 30}$ was selected for each problem instance.  A true solution matrix $X_0\in \reals^{30\times 10}$ was chosen with $7$ non-zero rows and random Gaussian entries in each non-zero row.  The data matrix was $B = AX_0+\eta,$ where $\eta$ was a random Gaussian noise matrix with $\sigma = 0.1.$   The solution was recovered with $\mu = 1.$

\paragraph{\textbf{Democratic Representations:}} 
We generated frames of dimension $500\times 1000$ by randomly selecting subset of rows from a unitary discrete Fourier transform matrix.  The signal $b$ was a complex-valued random Gaussian vector of length 500.  Equation \eqref{dem} was solved with $\mu = 300.$

\paragraph{\textbf{Matrix Completion:}}
We generated a $200\times1000$ matrix $X$ with i.i.d.\ random Gaussian entries of standard deviation $10.$
The SVD of $X$ was then computed, and all but the largest 5 singular values were set to zero to obtain a rank-$5$ matrix.  The sigmoidal logistic function was then applied element-wise to $X$ to obtain a matrix of probabilities; a random Bernoulli matrix $Y$ was drawn from the resulting distribution.  Problem \eqref{matcomp} was then solved using the logistic penalty function with $\mu=25.$

\paragraph{\textbf{Total Variation Denoising:}}
 The $256\times256$ Shepp-Logan phantom was constructed with pixel intensities that spanned the unit interval.  Test images were then contaminated with Gaussian noise ($\sigma=0.05$).  Images were denoised by solving \eqref{tv} with $\mu=0.1.$

\paragraph{\textbf{Support Vector Machine:}}
 Feature vectors were generated from two classes: one class containing random Gaussian variables with mean $-1$, and one with mean $+1.$  Each trial contained 1000 feature vectors of dimension 15.  A support vector machine was trained to separate the two classes with regularization parameter $C=10^{-2}.$

\paragraph{\textbf{Phase Retrieval:}}

We generated random vectors of length $N= 200$ with random Gaussian real and imaginary parts.  The set $\{a_i\}$ was created by randomly drawing complex Gaussian vectors.  The measurement vector $b$ of length $M = 600$ was creating by setting $b_i = |\la a_i,x\ra|^2,$ and then, contaminating the resulting measurements with real-valued Gaussian noise to have $\text{SNR}=13$\,dB.  Problem \eqref{smoothRankMin} was then solved with parameter $\mu=15,$ which was chosen to be large enough that the solution matrix had rank 1.

\paragraph{\textbf{Non-Negative Matrix Factorization:}}
 A rank-10 factorization problem was built from two matrices $X\in \reals^{800\times 10}$ and $Y\in \reals^{200\times 10}$ consisting of uniform random numbers from the interval $[0,1].$  The matrix $Q = XY^T$ was formed, and then contaminated with Gaussian noise of variance $10^{-2}.$  The matrix was then recovered by solving \eqref{nmf}.
 
 \paragraph{\textbf{Max-Norm Optimization:}}
 We solve the max-cut graph segmentation problem described in \cite{LRSST10}.  A ``two-moons'' data set was built with 1000 data points $\{x_i\},$ and the weighted adjacency matrix $W\in \reals^{1000\times 1000}$ was built with $W_{ij} =\delta-e^{\|x_i-x_j\|^2/\sigma^2},$ where $\delta = 0.01$ is a regularization parameter, and $\sigma=0.1$ is a scaling parameter.  Problem \eqref{maxcut} was solved using the weighted adjacency matrix.  The recovered matrix $X$ was then used to compute the labeling given by $\sign(Xr)$ where $r$ is a random Gaussian vector.    

\subsection{Results and Discussion}


 \begin{table}[t]
\centering
\caption{Complexity comparison of ``vanilla'' FBS, accelerated FBS (FISTA), and adaptive FBS (SpaRSA). We show the average number of iterations (and time per problem in seconds) to reach convergence. Adaptive FBS clearly outperforms other FBS variants for the considered test problems. \vspace{-5mm}} \label{kickasstable}
\vskip 0.15in
\begin{small}
\begin{tabular}{lccc}
\toprule
Problem & FBS & Accelerated &  Adaptive \\
\midrule
Lasso 100 &	  356 (0.138) & 	   55 (0.045) & 	   22 (0.021) \\
Lasso 500 &	   47 (0.065) & 	   20 (0.039) & 	    8 (0.018) \\
BPDN 100 &	  253 (0.088) & 	   48 (0.036) & 	   20 (0.022) \\
BPDN 500 &	   67 (0.077) & 	   23 (0.039) & 	   10 (0.019) \\
Logistic &	   40 (0.140) & 	   24 (0.091) & 	   14 (0.057) \\
MMV &	  657 (0.193) & 	   81 (0.046) & 	   58 (0.037) \\
Democratic &	   71 (0.074) & 	   31 (0.044) & 	   12 (0.028) \\
Mat Comp &	   69 (6.155) & 	   26 (2.569) & 	    8 (0.750) \\
TV Denoising &	 1000 (7.877) & 	  177 (1.800) & 	  102 (1.186) \\
SVM &	 3081 (0.643) & 	  244 (0.090) & 	   36 (0.024) \\
PhaseLift &	 1000 (36.720) & 	  186 (6.967) & 	   83 (3.614) \\
NMF &	 1000 (3.541) & 	  246 (1.319) & 	  173 (0.738) \\
Max-Norm &	  181 (16.578) & 	   43 (5.939) & 	   10 (0.909) \\
\bottomrule
\end{tabular}
\vspace{-2mm}
\end{small}
\end{table}

%

For all problems considered, both the accelerated and adaptive FBS dramatically out-performed ``vanilla'' FBS without these modifications.  However FBS does not reliably converge when adaptivity and acceleration are used simultaneously, and so the user must pick one.

To explore the efficiency of different approaches, we applied three variants of FBS to each test problem:  Plain FBS, FBS with acceleration (FISTA), and FBS with adaptive stepsizes (SpaRSA).   The accelerated FBS was accompanied by the restart rule \eqref{restart}, and all methods used backtracking line search. 
For each algorithm, the number of iterations (and time in seconds) needed to solve each test problem is reported in Table~\ref{kickasstable}.   

 For most problems considered,  the adaptive method out-performed the accelerated scheme by a factor of 3-to-5. For the SVM problem, the performance gap was somewhat larger.    We can take a closer look at the behavior of each method with the convergence curves in Figure \ref{fig:curves}.  Figure \ref{fig:curves}a shows the convergence for matrix completion, which looks typical for most problems including Lasso, penalized least-squares, logistic regression, and MMV.  For such problems, we see smooth, exponential decay of the error until machine precision is reached.  Note this empirical behavior is much better than the $O(1/k)$  worse-case global convergence bounds \cite{BT09}. 
 
 We also show some less-typical convergence curves, including SVM for which the performance gap between methods was extremely large.  The curves for non-negative matrix completion show more irregular behavior (including oscillations) because of non-convexity.   Finally, we show convergence curves for the total variation minimization in Figure \ref{fig:curves}d.  For this problem, adaptive and accelerated methods were competitive.  The adaptive method was superior in the low-precision regime, with the accelerated variant winning out in the high-precision regime.  This is largely because total variation involves the gradient operator, which has a large condition number.  Nesterov-type acceleration tends to be most effective for these types of poorly-conditioned problems, however the advantages over adaptivity are still slim in the examples shown here.

We note that the advantage of adaptivity is most pronounced for problems where $f$ is non-quadratic, i.e., for problems with a logistic data term.  In this case, the Hessian of~$f$ varies over the problem domain, and the optimal stepsize for FBS varies with it.  For such problems, an adaptive scheme is able to effectively match the stepsize to the local structure of the objective, resulting in fast convergence.


\section{Conclusion}
The forward-backward splitting method is a surprisingly simple way to solve a wide range of optimization problems.  Even seemingly complex problems involving non-differentiable objectives (total-variation, support vector machine, sparse regression), and complex constraint sets (semi-definite programing and max-norm regularization, etc...) can be reduced to a sequence of extremely simple steps.

Historically, FBS it is most commonly used for simple sparse regression problems despite its much wider applicability.  In many common domains, more complex splitting methods involving Lagrange multipliers (such as ADMM \cite{GL89, BPCPE10} and its variants \cite{EZC09,CP10}) are more commonly used.  However, these method are often more complex, memory intensive, and computationally burdensome than FBS.  Furthermore, FBS has a major advantage over other splitting methods -- algorithm parameters like stepsizes and stopping conditions are easily automated.  For other splitting methods, it is considerably more difficult to guarantee convergence for adaptive methods \cite{HYW00,GEB13}.
 For this reason it is fair to say that FBS is under-utilized for complex problems. 



\newpage

\appendix
\section{Convergence Proof for Non-Monotone Line Search}

We now consider the convergence of the backtracking line search discussed in section \ref{sec:linesearch}.
  \begin{theorem} \label{thm:converge}
  Suppose that FBS is applied to \eqref{general} with convex $g$ and differentiable $f$.  Suppose further that $h=f+g$ is   proper, lower semi-continuous, and has bounded level sets.     If $\{\tau^k\}$ is bounded below by a positive constant and  
 \eqn{backtrack}{
  f(x\kp)-  \hat f^k < \\ \la x\kp -x^k, \nabla f(x^k) \ra + \frac{1}{2\tau^k}  \| x\kp-x^k\|^2 
  }
  then
  $ \lim_{k\to\infty} h(x^k) = h\opt$,
  where $h\opt$ denotes the minimum value of $h$.
  \end{theorem}

\begin{proof}

  From the optimality condition (4), we have $0 \in \tau^k  \partial  g(x\kp) + x\kp-\bar x\kp $ and so
  $$0 = \tau^k G\kp +x\kp- \bar x\kp = \tau^k G\kp + x\kp-(x^k-\tau F^k) $$
  for some $G\kp \in \partial g(x\kp)$ and $F^k = \nabla f(x^k).$  From this we arrive at
  \aln{ \label{combinedOpt}
  x^k-x\kp = \tau^k (G\kp+F^k)
  }  
  and also
  \eqn{toNorm}{
   \la x\kp-x^k, F^k+G\kp \ra = - \frac{1}{\tau^k}\|x\kp-x^k\|^2.
   }
  
  Now, because $g$ is convex
  \aln{\label{convG}
    g(x^k) \ge g(x\kp)+\la x^k-x\kp, G\kp  \ra.
  }
  Subtracting \eqref{convG} from \eqref{backtrack} and applying \eqref{toNorm} yields
    \aln{ \label{telescope} 
 h(x\kp) =& \, f(x\kp)+g(x\kp)  \notag
 \\ \le & \, \hat f^k+ g(x^k) + \la x\kp-x^k, F^k+G\kp \ra \notag \\ 
                  & +\frac{1}{2\tau^k}\|x\kp-x^k\|^2 \notag \\
     = & \,  \hat f^k+ g(x^k)-\frac{1}{2\tau^k}\|x\kp-x^k\|^2\notag \\
      = &  \, \hat h^k-\frac{1}{2\tau^k}\|x\kp-x^k\|^2 .
  }
  where $\hat h^k =  \max \{ h^{k-1},h^{k-2},\dots,h^{k- \min\{M,k\} }  \} .$  Note that $\{\hat h^k\}$ is a monotonically decreasing bounded sequence, and thus has a limit $\hat h^*.$  
  
  We aim to show that  $\hat h\opt$ is the minimal value of  $h$.   Observe that $\hat h^k = h^{k'}$ for some $k'$ with  $k-M\le k' \le k.$   It is clear from \eqref{telescope} that there must exist a sub-sequence  $\{ x^{k(i)} \}$ with $h(x^{k(i)} ) = \hat h^k$ such that 
  \aln{ \label{diffToZero}
  \lim_{i\to\infty} \frac{1}{2\tau^k}\|x^{k(i)+1} -x^{k(i)} \|^2=0
  }
   (otherwise, equation \eqref{telescope} would imply $\hat h^k\to -\infty$ for large $k$).  Note that  the level sets of $h$ are assumed to be bounded, and by equation \eqref{telescope} $\{h(x^k)\}$ is bounded as well.  By compactness,  we may assume without loss of generality that $\{ x^{k(i)} \}$ is a convergent sub-sequence of iterates with limit point $x\opt$.  
   
   Now, from  \eqref{combinedOpt}, we have 
   \aln{ \label{almostGradient}
    \frac{1}{2\tau^k} \|x^{k(i)+1} -x^{k(i)} \|^2 = \frac{\tau^k}{2} \| G^{k(i)+1} +F^{k(i)} \|^2 .
    }
   Equation \eqref{almostGradient}, together with  \eqref{diffToZero} and the fact that $\tau^k$ is bounded away from zero, implies that
     $$\lim_{i\to\infty} \| G^{k(i)+1} +F^{k(i)} \|^2 =0.$$   Because $\nabla f$ is Lipschitz continuous and $\|x^{k(i)+1} -x^{k(i)} \|\to0$, we also conclude that  $ x^{k(i)+1} \to x\opt$ and
        \aln{ \label{gradToZero}
        \lim_{i\to\infty} \| G^{k(i)+1} +F^{k(i+1)} \|^2 =0. 
        }
  Note that $G^{k(i)+1} +F^{k(i+1)} \in \partial h(x^{k(i)+1})$.   Because the sub-differential of a convex function is continuous\footnote{More formally, the convex  sub-differential is a multi-valued function which is upper semicontinous in a topological sense. See \cite{CSZ07,BC11}.  },  \eqref{gradToZero} implies that $0 \in \partial h(x\opt),$ and so  $x\opt$ is a minimizer of $h$.
  
    We have shown that $\lim_{k\to\infty} \hat h^k = h(x\opt) =  h\opt.$ Because $h\opt \le h(x^k)\le \hat h^k,$ we arrive at the conclusion 
     $$\lim_{k\to\infty} h(x^k) = h^*.$$
    
    \end{proof}

\bibliography{/Users/Tom/Documents/latexDocs/bib/tom_bibdesk}

\end{document}